\documentclass{amsart}
\usepackage{amssymb,stmaryrd,mathrsfs}
\usepackage[all]{xy}
\usepackage{graphicx}
\usepackage{verbatim}
\usepackage{tikz}

\let\oldmarginpar\marginpar
\renewcommand\marginpar[1]{\-\oldmarginpar[\raggedleft\footnotesize #1]%
{\raggedright\footnotesize #1}}

\usepackage[hang]{caption}
\usepackage{mathtools}
\usepackage{manfnt}

\usepackage{color}
\newtheorem{thm}{Theorem}[section]
\newtheorem{lem}[thm]{Lemma}
\newtheorem{prop}[thm]{Proposition}

\theoremstyle{definition}

\theoremstyle{remark}
\newtheorem{rmk}[thm]{Remark}

\newcommand{\Z}{\mathbb{Z}}
\newcommand{\R}{\mathbb{R}}

\newcommand{\Dn}{\Delta_n}
\newcommand{\Qn}{\square_n}
\newcommand{\Xn}{\diamondsuit_n}
\newcommand{\sF}{\mathscr{F}}
\newcommand{\sP}{\mathcal{P}}
\newcommand{\sFstd}{\mathscr{F}^\text{std}}

\DeclareMathOperator{\rank}{\mathrm{rank}}

\title[The absolute orders on the Coxeter groups $A_n$ and $B_n$ are Sperner]{The absolute orders on the Coxeter groups $A_n$ and $B_n$ are Sperner}

\author{Lawrence~H. Harper}
\address{Department of Mathematics\\
  University of California\\
  Riverside, CA 92521} 
\email{harper@math.ucr.edu}
\author{Gene~B. Kim}
\address{Department of Mathematics\\
  University of Southern California\\
  Los Angeles, CA 90089} 
\email{genebkim@math.usc.edu} 
\author{Neal Livesay}
\address{Department of Mathematics\\
  University of California\\
  Riverside, CA 92521} 
\email{neall@ucr.edu} 

\begin{document}
\begin{abstract}
Over 50 years ago, Rota posted the following celebrated ``Research Problem'': prove or disprove that the partial order of partitions on an $n$-set (i.e., the refinement order) is Sperner.  A counterexample was eventually discovered by Canfield in 1978.  However, Harper and Kim recently proved that a closely related order --- i.e., the refinement order on the symmetric group --- is not only Sperner, but strong Sperner.  Equivalently, the well-known absolute order on the symmetric group is strong Sperner.  In this paper, we extend these results by giving a concise, elegant proof that the absolute orders on the Coxeter groups $A_n$ and $B_n$ are strong Sperner.
\end{abstract}
\maketitle

\section{Introduction}
In 1928, Sperner \cite{Sp} proved that the poset of subsets of $[n]=\{1,2,\ldots,n\}$ has the property that none of its antichains (i.e., collections of pairwise incomparable vertices in the poset) has cardinality larger than the largest rank.  In 1967, Rota \cite{Ro} famously conjectured that the refinement order $\Pi_n$ (i.e., the poset of partitions of $[n]$) has this same property (which became known as the \emph{Sperner property}).  There were many attempts to prove Rota's conjecture, but in 1978, Canfield \cite{Can} discovered a counterexample to Rota's conjecture for $n$ larger than Avogadro's number.  Although the refinement order $\Pi_n$ is not Sperner, there is a closely related poset on the symmetric group $S_n$ (also called the refinement order) which Harper and Kim \cite{HK} recently proved is not only Sperner, but strong Sperner.  The refinement order on $S_n$ is anti-isomorphic to a well-known (see, e.g., \cite{Arm}) order on $S_n$ called the \emph{absolute order}; i.e., $x\leq y$ in the refinement order if and only if $y\leq x$ in the absolute order.  Hence an immediate corollary to \cite{HK} is that the absolute order on $S_n$ is strong Sperner.

\begin{figure}
\centering
\begin{tikzpicture}
\def\ra{5.5}
\def\het{.2}
\draw[fill=black] (-.8,0+\het) circle (3pt);
\draw[fill=black] (.8,0+\het) circle (3pt);
\draw[fill=black] (-1.6,1.5) circle (3pt);
\draw[fill=black] (0,1.5) circle (3pt);
\draw[fill=black] (1.6,1.5) circle (3pt);
\draw[fill=black] (0,3-\het) circle (3pt);
\node at (-1,-.4+\het) {$(1\text{ } 2\text{ } 3)$};
\node at (.7,-.4+\het) {$(1\text{ } 3\text{ } 2)$};
\node at (-2.1,1.5) {$(1\text{ } 2)$};
\node at (-.5,1.5) {$(1\text{ } 3)$};
\node at (2.1,1.5) {$(2\text{ } 3)$};
\node at (-.3,3.05-\het) {$e$};
\draw[thick] (-.8,0+\het) -- (-1.6,1.5) -- (0,3-\het) -- (0,1.5) -- (-.8,0+\het) -- (1.6,1.5) -- (0,3-\het);
\draw[thick] (.8,0+\het) -- (-1.6,1.5) -- (0,3-\het) -- (0,1.5) -- (.8,0+\het) -- (1.6,1.5) -- (0,3-\het);

\draw[fill=black] (0+\ra,0+\het) circle (3pt);
\draw[fill=black] (-1.6+\ra,1.5) circle (3pt);
\draw[fill=black] (0+\ra,1.5) circle (3pt);
\draw[fill=black] (1.6+\ra,1.5) circle (3pt);
\draw[fill=black] (0+\ra,3-\het) circle (3pt);
\node at (0.5+\ra,0+\het) {$123$};
\node at (-1+\ra,1.5) {$12|3$};
\node at (0.5+\ra,1.5) {$13|2$};
\node at (2.1+\ra,1.5) {$23|3$};
\node at (0.5+\ra,3-\het) {$1|2|3$};
\draw[thick] (0+\ra,0+\het) -- (-1.6+\ra,1.5) -- (0+\ra,3-\het) -- (0+\ra,1.5) -- (0+\ra,0+\het) -- (1.6+\ra,1.5) -- (0+\ra,3-\het);

\end{tikzpicture}
\caption{The refinement orders on $S_3$ and $\Pi_3$ respectively.}
\label{refinementorders}
\end{figure}
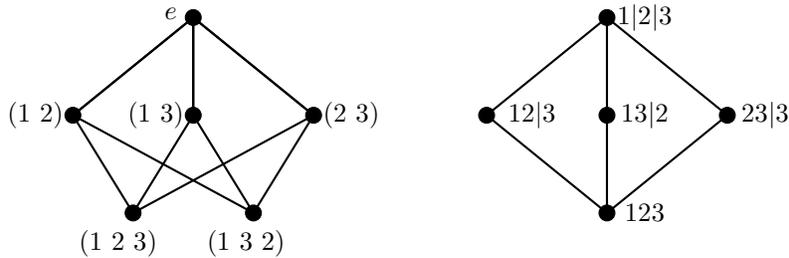

The main result in this paper is Theorem~\ref{mainthm}, which states that the absolute orders on the Coxeter groups $A_n$ and $B_n$ are strong Sperner.  The outline of the paper is as follows.  In Section~\ref{regpoly}, we recall that the Coxeter groups $A_n$ and $B_n$ can be realized as symmetry groups for regular polytopes.  This perspective is crucial to the proof of our main result.  In Section~\ref{posets}, we recall basic definitions and facts regarding posets, the Sperner property, and absolute orders.  Section~\ref{factorsection} contains Proposition~\ref{factor}, which states that any element of $A_n$ or $B_n$ can be uniquely expressed as a product of certain reflective symmetries of a fixed flag.  This proposition is a key ingredient in the proof of our main theorem, but is also an interesting result independently.  Another key ingredient is Harper's Product Theorem \cite{LH}, which is used in Section~\ref{NFP} to show that certain spanning subposets of the absolute orders have the normalized flow property, a strengthening of the strong Sperner property.  Finally, Section~\ref{mainsection} contains the proof of the main theorem.

\section{The regular $n$-simplex and $n$-cube and their symmetries}\label{regpoly} We begin by recalling some basic facts about the regular $n$-simplex and $n$-cube and their symmetries.  The regular $n$-simplex $\Dn$ is the convex hull of the standard basis $\{e_1, e_2,\ldots, e_{n+1}\}$ for $\R^{n+1}$.  The symmetry group of $\Dn$ is the finite irreducible Coxeter group $A_n$.  There is a bijective correspondence between faces of $\Dn$ and subsets of $[n+1]=\{1,2,\ldots,n+1\}$ determined by mapping the $0$-dimensional face $\{e_i\}$ in $\Dn$ to the singleton $\{i+1\}$ for each $i$.  In this correspondence, the $i$-dimensional faces (or \emph{$i$-faces}) in $\Dn$ correspond to subsets of $[n+1]$ of size $i+1$, and the \emph{facets} (i.e., the $(n-1)$-faces) correspond to subsets of size $n$.  There is an isomorphism between $A_n$ and $S_{n+1}$ such that the action of $A_n$ on $\Dn$ corresponds to the natural action of $S_{n+1}$ on $[n+1]$.  The reflections in $A_n$ correspond to the transpositions in $S_{n+1}$; in particular, the reflection swapping $e_i$ and $e_j$ corresponds to the transposition $(i\text{ } j)$.

The $n$-cube $\Qn$ is the convex hull in $\R^n$ of the Cartesian product $\{-1,1\}^n\subset\R^n$.  The dual polytope to the $n$-cube is the $n$-cross-polytope $\Xn$, which is the convex hull of $\{\pm e_1,\pm e_2,\ldots, \pm e_n\}\subset \R^n$.  The map $e_i\mapsto i$ induces a bijective correspondence between the faces of $\Xn$ and subsets of $\{\pm 1, \pm 2, \ldots, \pm n\}$ satisfying a special property.  Specifically, the $i$-faces of $\Xn$ correspond to subsets $S\subset\{\pm j\}_{j=1}^n$ of size $i$, with the property that $k\in S$ implies $-k\notin S$ for all $k\in \{\pm j\}_{j=1}^n$.  The symmetry group for each of the dual polytopes $\Qn$ and $\Xn$ is the hyperoctahedral group $B_n$.  The hyperoctahedral group is isomorphic to the \emph{signed permutation group} --- i.e., the group of permutations $w$ of the set $\{\pm j\}_{j=1}^n$ with the property that $w(-i)=-w(i)$ for all $i$ --- and there is an obvious choice of isomorphism between these two groups such that their respective actions on $\Xn$ and $\{\pm j\}_{j=1}^n$ correspond.  Following \cite[Section 2.2]{MK}, we denote the signed permutation with cycle form $(a_1\text{ } a_2\text{ } \cdots \text{ }a_k)(-a_1\text{ } -a_2\text{ } \cdots\text{ } -a_k)$ by $(\!(a_1, a_2, \ldots, a_k)\!)$, and $(a_1\text{ } a_2\text{ } \cdots\text{ } a_k\text{ } -a_1\text{ } -a_2\text{ } \cdots \text{ }-a_k)$ by $[a_1, a_2, \ldots, a_k]$.  The set of reflections in $B_n$ corresponds to the union of $\{[i]\}_{i=1}^n$ and $\{(\!(i,j)\!), (\!(i,-j)\!)\}_{1\leq i<j\leq n}$; in particular, the reflection of $\Xn$ swapping $e_i$ and $-e_i$ corresponds to $[i]$, and the reflection swapping $e_i$ and $\pm e_j$ (for $i\neq j$) corresponds to $(i\text{ } \pm j)$.

\begin{lem}\label{uniquerefl}
For any pair $(C,C^\prime)$ of distinct facets in $\Dn$ (resp. $\Qn$), there is a unique reflection in $A_n$ (resp. $B_n$) mapping $C$ to $C^\prime$.
\end{lem}
\begin{proof}
Let $C\neq C^\prime$ be facets in $\Dn$.  Recall that facets in $\Dn$ correspond to subsets of size $n$ in $[n+1]$, and reflective symmetries of $\Dn$ correspond to transpositions in $S_{n+1}$.  Since $C\neq C^\prime$, it follows that $C-C^\prime = \{i\}$ and $C^\prime -C = \{j\}$ for some $i\neq j$.  The unique reflection mapping $C$ to $C^\prime$ is $(i\text{ } j)$.

Now let $C\neq C^\prime$ be facets in $\Qn$.  The facets of $\Qn$ correspond to the vertices of $\Xn$, which in turn correspond to elements of $\{\pm 1, \pm 2, \ldots, \pm n\}$.  Suppose without loss of generality that $C$ corresponds to $1$.  Either $C^\prime$ corresponds to $-1$, $j$ for some $j\neq 1$, or $-j$ for some $j\neq 1$.  In any case, there is a unique reflection in $B_n$ mapping $C$ to $C^\prime$ (specifically, the reflections $[1]$, $(\!(1,j)\!)$, and $(\!(1,-j)\!)$, respectively).
\end{proof}

Define a \emph{(complete) flag} $\sF=(\sP_i)_{i=0}^n$ in an $n$-dimensional regular polytope $\sP$ to be a sequence of faces in $\sP$, ordered by containment, with $\dim(\sP_i)=i$.  The action of $A_n$ (resp. $B_n$) on $\Dn$ (resp. $\Qn$) induces a simply transitive action on the associated set of flags.  Hence if we designate some flag in $\Dn$ or $\Qn$ --- call it the \emph{standard flag} $\sFstd=(\sP_i^\text{std})_{i=0}^n$ --- then a correspondence between elements of its symmetry group and its set of flags can be defined via $w\mapsto w\cdot\sFstd$.  Note that, for all $i\in [0,n]$, the $i$-faces for the $n$-simplex (resp. the $n$-cube) are $i$-simplices (resp. $i$-cubes).

\section{Posets, the Sperner property, and the absolute orders}\label{posets} Let $P$ be a finite graded poset with rank decomposition $P=\bigsqcup_{i=0}^r P_i$.  In this paper, the term ``poset'' is always used to mean finite graded poset.  A \emph{$k$-family} in $P$ is a subset of $P$ containing no chain of size $k+1$.  The poset $P$ is defined to be \emph{$k$-Sperner} if the union of the $k$ largest rank levels $P_i$ is a $k$-family of maximal size; \emph{strong Sperner} if $P$ is $k$-Sperner for all $k\in [1,r+1]$; and \emph{rank unimodal} if $|P_0|\leq |P_1|\leq \cdots \leq |P_{j-1}|\leq |P_j| \geq |P_{j+1}|\geq\cdots\geq |P_r|$ for some $j$.  Note that the 1-Sperner property is otherwise known as the Sperner property, and a 1-family is otherwise known as an antichain.  A \emph{spanning subposet} for a poset $P$ is any subposet with the same vertex set and rank function as $P$.  
\begin{lem}\label{subposet}
Suppose that $P$ is a spanning subposet of $P^\prime$.  If $P$ is rank unimodal and strong Sperner, then so is $P^\prime$.  
\end{lem}
\begin{proof}
Since $P$ is rank unimodal, its largest $k$ rank levels can be chosen so that their ranks are consecutive.  Their union is a $k$-family in both $P$ and $P^\prime$.  Since $P$ is $k$-Sperner, this union is a $k$-family in $P$ of maximal size, and therefore a $k$-family in $P^\prime$ of maximal size.
\end{proof}

We briefly recall some generalities about absolute orders; see, e.g., \cite{Arm} for details.  Let $W$ be a finite Coxeter group with set of reflections $T$.  The \emph{absolute length} $l_T$ on $W$ is the word length with respect to $T$.  The \emph{absolute order} on $W$ is defined by \[\pi\leq \mu\text{ if and only if }l_T(\mu)=l_T(\pi)+l_T(\pi^{-1}\mu)\] for all $\pi, \mu\in W$.  Equivalently, the absolute order is the partial order on $W$ generated by the covering relations $w\rightarrow tw$, where $w\in W$, $t\in T$, and $l_T(w)<l_T(tw)$.  This order is graded with rank function $l_T$.  The absolute length generating function $\mathrm{P}_W(q)=\sum_{w\in W}q^{l_T(w)}$ satifies $\mathrm{P}_W(q)=\prod_{i=1}^n (1+(d_i-1)q)$, where $(d_i)_{i=1}^n$ is the degree sequence for $W$ (and $n=\rank(W)$) \cite[p. 35]{Arm}.  It follows that $|T|=|l_T^{-1}(1)|=\sum_{i=1}^n (d_i -1)$.  Moreover, the rank sequence $\left(|l^{-1}_T(i)|\right)_{i=0}^{n}$ for any absolute order is strictly log-concave by \cite[Theorem 4.5.2]{HSW}, and thus all of the absolute orders are rank unimodal.  

\section{Factoring elements of $A_n$ and $B_n$}\label{factorsection}
For all that follows, $\sP$ denotes the regular $n$-simplex or $n$-cube, and $W$ denotes the corresponding symmetry group.  Any reflective symmetry of an $i$-face $\sP_i$ of $\sP$ uniquely extends to a reflective symmetry of $\sP$.  Define $T_{\sP_i}$ to be the embedding of the set of reflections of $\sP_i$ into $W$.  

\begin{lem}\label{difference}
Let $\sP$ be the $n$-simplex or $n$-cube, and let $W$ be the corresponding group of symmetries with degree sequence $(d_i)_{i=1}^n$.  Fix a standard flag $(\sP_i^{\text{std}})_{i=0}^n$ in $\sP$, and set $T_i=T_{{\sP_i}^{\text{std}}}$.  It follows that,  for all $i\in [1,n]$, $|T_i-T_{i-1}|=d_i-1$.
\end{lem}

\begin{proof}
The $n$-simplex (resp. the $n$-cube) has the property that, for each $i$, each of its $i$-faces is an $i$-simplex (resp. $i$-cube).  Hence the symmetry group for any of its $i$-faces is $A_i$ (resp. $B_i$).  If the degree sequence for the $n$-simplex (resp. $n$-cube) is $(d_j)_{j=1}^n$, then the degree sequence associated to an $i$-face is $(d_j)_{j=1}^i$.  It follows that $|T_i-T_{i-1}|=|T_i|-|T_{i-1}|=\sum_{j=1}^i (d_j-1)-\sum_{j=1}^{i-1} (d_j-1)=d_i-1$.
\end{proof}

\noindent It is easily verified that the relation between a regular polytope and the degree sequence of its symmetry group described in Lemma~\ref{difference} is satisfied by \emph{precisely} the $n$-simplices, $n$-cubes, and $m$-gons (and none of the other regular polytopes).  For ease of reference, we note here that the degree sequence $(d_i)_{i=1}^n$ for $A_n$ is defined by $d_i=i+1$, and for $B_n$ is defined by $d_i=2i$.  
\begin{prop}\label{factor}
Let $\sP$ be the $n$-simplex or $n$-cube, and let $W$ be the associated symmetry group.  Fix a standard flag $\sFstd=(\sP_i^{\text{std}})_{i=0}^n$ in $\sP$, and set $T_i=T_{\sP_i^{\text{std}}}$.  
\begin{enumerate}
\item Any element $w\in W$ has a unique factorization of the form \[w=r_n r_{n-1} \cdots r_2r_1\] with $r_i\in (T_i-T_{i-1}) \sqcup \{e\}$ for each $i$, where $e$ is the identity in $W$.
\item Given such a factorization, the length can be computed via \[ l_T\left(\prod_{i=0}^{n-1}r_{n-i}\right)=|\{i: r_i\neq e\}|.\]
\item Finally, $\prod_{i=0}^{n-1}r_{n-i}$ covers $\prod_{i=0}^{n-1}r^\prime_{n-i}$ if there exists $k$ such that $r_k\neq r^\prime_k=e$ and $r_j=r_j^\prime$ for all $j\neq k$.
\end{enumerate}
\end{prop}

\begin{proof}
We begin by proving (1).  The claim is clearly true for $n=1$.  Now let $n>1$ be arbitrary, and suppose the claim is true for $n-1$.  Let $w\in W$, with corresponding flag $\sF=(\sP_i)_{i=0}^{n}$.  If $(\sP_i)_{i=0}^{n-1}$ is a flag in the ``standard facet'' $\sP_{n-1}^{\text{std}}$, then the claim follows by the inductive hypothesis.  Suppose instead that $(\sP_i)_{i=0}^{n-1}$ is a flag in some other facet $C$.  Lemma~\ref{uniquerefl} implies that there is a unique reflection $r_n\in (T_n-T_{n-1})-\{e\}$ mapping $C$ to $\sP_{n-1}^{\text{std}}$.  By the inductive hypothesis, it follows that $r_n\cdot \sF=(r_{n-1}\cdots r_2r_1)\cdot\sFstd$ with $r_i\in (T_i-T_{i-1})\sqcup \{e\}$ for all $i\in [1,n-1]$.  Therefore $w\cdot\sFstd=\sF=r_n(r_{n-1}\cdots r_2r_1)\cdot\sFstd$, and the claim follows.

To prove (2), we first let $\sP$ be the $n$-simplex and let $W$ be its symmetry group.  Assume without loss of generality that $\sFstd=\left([i+1]\right)_{i=0}^n$.  Then $T_i-T_{i-1}$ consists of all transpositions $(j\text{  }(i+1))$ with $j\in [1,i]$.  If $w$ is a product of elements in $T_{i-1}\sqcup \{e\}$, then $w$ is a permutation of $[i]$.  Hence $l_T(r_iw)>l_T(w)$, which implies that $l_T(r_iw)=l_T(w)+1$.  The claim follows from a straight-forward induction on $n$.  Now let $\sP$ be the $n$-cube and $W$ its symmetry group.  Assume without loss of generality $\sFstd=(\sP_i^{\text{std}})_{i=0}^n$ is chosen so that the symmetries of $\sP^{\text{std}}_i$ correspond to symmetries of $\{\pm 1, \pm 2, \ldots, \pm i\}$.  Then $T_i$ consists of all reflections of the form $[j]$, $(\!(j, k)\!)$, and $(\!(-j, k)\!)$ with $j, k\in [1, i]$ and $j\neq k$, and $T_i-T_{i-1}$ consists of all reflections of the form $[i]$, $(\!(j, i)\!)$, and $(\!(-j, i)\!)$ with $j\in [1,i-1]$.  Similar to the case above, the product $r_iw$ of $r_i$ in $T_i-T_{i-1}$ with a product $w$ of reflections in $T_{i-1}$ has $l_T(r_iw)>l_T(w)$.  Hence $l_T(r_iw)=l_T(w)+1$, and the claim follows by induction.

Finally, to prove (3), let $w$ and $w^\prime$ be elements of $W$ with the property that their expansions $w=\prod_{i=0}^{n-1} r_i$ and $w^\prime=\prod_{i=0}^{n-1} r^\prime_i$ satisfy $r_k\neq r^\prime_k=e$ for some $k$ and $r_j=r_j^\prime$ for all $j\neq k$.  By Proposition~\ref{factor}.2, it follows that $l_T(w^\prime)+1=l_T(w)$.  Set $\sigma=\prod_{i=0}^{k-1} r_i$ and $\tau=\prod_{i=k+1}^{n-1} r_i$, so that $w=\sigma r_k \tau$ and $w^\prime = \sigma\tau$.  Then $l_T((w^\prime)^{-1}w)=l_T(\tau^{-1}\sigma^{-1}\sigma r_k\tau)=l_T(\tau^{-1}r_k\tau)=1$.  Since $l_T(w^\prime)+l_T((w^\prime)^{-1}w)=l_T(w)$, it follows that $w$ covers $w^\prime$.\end{proof}

\section{The Sperner property for products}\label{NFP} It is not the case that a product of Sperner (or even strong Sperner) posets is necessarily Sperner.  However, if two posets satisfy a strengthening of the strong Sperner property called the \emph{normalized flow property} (abbreviated NFP), as well as an additional easily verified hypothesis, then their product has the NFP (and hence is strong Sperner) by Harper's Product Theorem \cite{LH}.  As the NFP is only referenced in this paper for the purpose of proving Lemma~\ref{prodstar}, we do not recall the fundamental definitions and theory regarding flows on posets here.  The curious reader may refer to, e.g., \cite{LH}.

Define a \emph{$k$-claw} $C_k=\bigsqcup_{l=0}^1(C_k)_l$ to be the graded poset with $|(C_k)_0|=1$, $|C_k|=k$, and whose underlying graph is complete bipartite.  

\begin{lem}\label{prodstar}
Let $\{k_i\}_{i=1}^n\subset \Z_+$.  The product poset $\prod_{i=1}^nC_{k_i}$ is strong Sperner.
\end{lem}
\begin{proof} 
Suppose that the weight $\nu (x)$ of any vertex $x$ in a claw equals 1.  Each claw $C_{k_i}$ has a normalized flow (e.g., the flow $f$ defined by setting $f(e)=1/(k_i-1)$ for all edges $e$) and thus has the NFP.  Moreover, the sequence $\left(\nu((C_{k_i})_l)\right)_{l=0}^1$  is clearly log-concave (or 2-positive using the terminology of \cite{LH}) for each $i$.  By \cite[Product Theorem]{LH}, it follows that any product of claws $C_{k_i}\times C_{k_j}$ has NFP and satisfies the property that the sequence $\left(\nu((C_{k_i}\times C_{k_j})_l)\right)_{l=0}^2$ is log-concave.  The claim follows by induction.
\end{proof}

\section{Main result}\label{mainsection}
\begin{thm}\label{mainthm}
The absolute orders on $A_n$ and $B_n$ are strong Sperner.
\end{thm}

\begin{proof}
Let $\sP$ be the $n$-simplex or $n$-cube, and let $W$ be the associated symmetry group.  Fix a standard flag $\sFstd=(\sP_i^{\text{std}})_{i=0}^n$ in $\sP$, and set $T_i=T_{\sP_i^{\text{std}}}$.  Let $(d_i)_{i=1}^n$ be the degree sequence for $W$.  Consider the product poset \[\prod_{i=0}^{n-1} C_{d_{n-i}} = C_{d_n}\times \cdots \times C_{d_2}\times C_{d_1}\] of claws $C_{d_i}$.  For each $i$, define a bijective correspondence between the vertices of the claw $C_{d_i}$ and the elements of $(T_i-T_{i-1})\sqcup \{e\}$ by mapping the $d_i-1$ vertices in $(C_{d_i})_1$ bijectively onto $T_i-T_{i-1}$ (such a bijection exists by Lemma~\ref{difference}) and the rank $0$ vertex in $C_{d_i}$ to $e$.  These bijective correspondences between claws and sets of reflections induce a bijective correspondence $\phi(r_n, \ldots, r_2,r_1)=r_n\cdots r_2r_1$ between the vertices of the product poset $\prod_{i=0}^{n-1} C_{d_{n-i}}$ and the vertices of the absolute order $W$ by Proposition~\ref{factor}(1).

We claim that $\prod_{i=0}^{n-1} C_{d_{n-i}}$ can be viewed as a spanning subposet of $W$ via the above bijection between of the vertex sets.  It suffices to prove that if $y$ covers $x$ in $\prod_{i=0}^{n-1} C_{d_{n-i}}$, then $\phi (y)$ covers $\phi (x)$ in $W$.  Suppose that $(r_n,\ldots,r_2,r_1)$ covers $(r^\prime_n,\ldots,r^\prime_2,r^\prime_1)$ in the product of claws.  Then there exists $k$ for which $r_k\neq r_k^\prime=e$ and $r_j=r_j^\prime$ for all $j\neq k$.  By Proposition~\ref{factor}(3), the claim immediately follows.  By Lemma~\ref{prodstar}, $\prod_{i=0}^{n-1} C_{d_{n-i}}$ is strong Sperner.  Since $\prod_{i=0}^{n-1} C_{d_{n-i}}$ is a spanning subposet of $W$, it follows by Lemma~\ref{subposet} that $W$ is strong Sperner.
\end{proof}

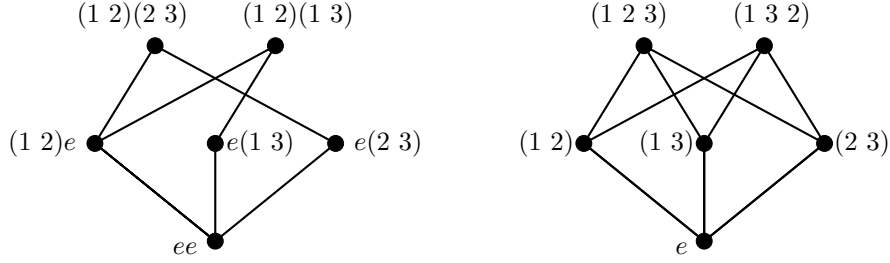
\begin{figure}
\centering
\begin{tikzpicture}
\def\ra{6.5}
\def\het{.2}
\draw[fill=black] (-.8,3-\het) circle (3pt);
\draw[fill=black] (.8,3-\het) circle (3pt);
\draw[fill=black] (-1.6,1.5) circle (3pt);
\draw[fill=black] (0,1.5) circle (3pt);
\draw[fill=black] (1.6,1.5) circle (3pt);
\draw[fill=black] (0,0+\het) circle (3pt);
\node at (-1.1,3.4-\het) {$(1\text{ } 2)(2\text{ } 3)$};
\node at (1.1,3.4-\het) {$(1\text{ }2)(1\text{ } 3)$};
\node at (-2.3,1.5) {$(1\text{ } 2)e$};
\node at (.6,1.5) {$e(1\text{ } 3)$};
\node at (2.3,1.5) {$e(2\text{ } 3)$};
\node at (-.4,-.1+\het) {$ee$};
\draw[thick] (-.8,3-\het) -- (-1.6,1.5) -- (0,0+\het) -- (0,1.5) -- (.8,3-\het) -- (-1.6,1.5) -- (0,0+\het);
\draw[thick] (-.8,3-\het) -- (1.6,1.5) -- (0,0+\het);

\draw[fill=black] (-.8+\ra,3-\het) circle (3pt);
\draw[fill=black] (.8+\ra,3-\het) circle (3pt);
\draw[fill=black] (-1.6+\ra,1.5) circle (3pt);
\draw[fill=black] (0+\ra,1.5) circle (3pt);
\draw[fill=black] (1.6+\ra,1.5) circle (3pt);
\draw[fill=black] (0+\ra,0+\het) circle (3pt);
\node at (-1+\ra,3.4-\het) {$(1\text{ } 2\text{ } 3)$};
\node at (.9+\ra,3.4-\het) {$(1\text{ } 3\text{ } 2)$};
\node at (-2.1+\ra,1.5) {$(1\text{ } 2)$};
\node at (-.5+\ra,1.5) {$(1\text{ } 3)$};
\node at (2.1+\ra,1.5) {$(2\text{ } 3)$};
\node at (-.3+\ra,-.1+\het) {$e$};
\draw[thick] (-.8+\ra,3-\het) -- (-1.6+\ra,1.5) -- (0+\ra,0+\het) -- (0+\ra,1.5) -- (-.8+\ra,3-\het) -- (1.6+\ra,1.5) -- (0+\ra,0+\het);
\draw[thick] (.8+\ra,3-\het) -- (-1.6+\ra,1.5) -- (0+\ra,0+\het) -- (0+\ra,1.5) -- (.8+\ra,3-\het) -- (1.6+\ra,1.5) -- (0+\ra,0+\het);

\end{tikzpicture}
\caption{The product of claws $C_{d_1}\times C_{d_2}$ can be viewed as a spanning subposet of the absolute order on $A_3$.}
\label{productandabsolute}
\end{figure}

\begin{rmk}
The remaining finite irreducible Coxeter groups (see, e.g., \cite[p. 2]{Arm}) are the other regular polytope symmetry groups $I_2(m)$ (for $m\ge 5$), $H_3$, $H_4$, and $F_4$, and the Weyl groups $D_n$ (for $n\ge 4$), $E_6$, $E_7$, and $E_8$.  It is straight-forward to verify that Lemma~\ref{uniquerefl},  Lemma~\ref{difference}, Proposition~\ref{factor}, and Theorem~\ref{mainthm} extend to the dihedral groups $I_2(m)$ for all $m$.  It follows that the absolute orders on the dihedral groups are strong Sperner.
\end{rmk}

\end{document}